\documentclass[12pt]{amsart}

\usepackage{amssymb}
\usepackage{mathrsfs}
\usepackage{hyperref}

\theoremstyle{plain}
\newtheorem{theorem}{Theorem}[section]
\newtheorem{proposition}[theorem]{Proposition}
\newtheorem{lemma}[theorem]{Lemma}
\newtheorem{corollary}[theorem]{Corollary}

\theoremstyle{definition}
\newtheorem{definition}{Definition}[section]
\newtheorem{example}{Example}[section]

\theoremstyle{remark}
\newtheorem*{remark}{Remark}
\title{Galois Extensions via Finiteness of Orbits}
\author{Nikolaos Marmaridis}
\address{Department of Mathematics, University of Ioannina\newline
Ioannina 451 10, Greece}
\email{nmarmar@uoi.gr}

\subjclass[2020]{Primary: 12F10; Secondary  12F05, 20B25}

\keywords{Galois extensions, fixed fields, group actions, finite orbits, simple extensions}

\begin{document}
\begin{abstract}
We present an orbit--theoretic reformulation of Galois theory based on the natural
action of automorphism groups on fields.
Given a field $\mathbf{E}$ and a subgroup $H$ of the automorphism group $\mathrm{Aut}(\mathbf{E})$, we show
that algebraic properties of the extension $\mathbf{E}/\mathbf{E}^H$, where $\mathbf{E}^H$ denotes the fixed field of $H$,  are encoded in the $H$-orbits arising from the action of $H$ on $\mathbf{E}$.

An element $\alpha \in \mathbf{E}$ is algebraic over $\mathbf{E}^H$ if and only if its $H$--orbit is finite. In that case, its minimal polynomial can be explicitly constructed as the product of linear factors over its orbit --a construction that also ensures separability.

At the level of field extensions, we prove that $\mathbf{E}/\mathbf{E}^H$ is Galois if and only if all $H$--orbits have finite length, and that $\mathbf{E}/\mathbf{E}^H$ is a finite Galois extension if and only if the lengths of the $H$--orbits are bounded above. This provides a unified orbit--theoretic characterization of algebraicity, separability, normality, and degree. Artin's Lemma is recovered as a direct consequence of this framework.

Finally, we show that for simple extensions, the fixed field under a subgroup $H$ of $\mathrm{Aut}(\mathbf{F}(\alpha)/\mathbf{F})$ can be described explicitly by evaluating elementary symmetric polynomials on the $H$--orbit of $\alpha$, provided this orbit is finite. This leads to an effective method for computing fixed fields directly from orbit data. A classical example is included to illustrate the approach.
\end{abstract}

\maketitle

\section{Introduction}
Classical Galois theory analyzes field extensions through the symmetries of polynomial roots and their corresponding automorphism groups. A fundamental result, see \cite[Theorem~7.31]{margal},  and \cite[Theorem~4.9]{Morandi}, asserts that a field extension $\mathbf{E}/\mathbf{F}$ is Galois if and only if it is both normal and separable.

Rather than approaching these conditions via splitting fields and minimal polynomials, we adopt a structural perspective based on group actions. Specifically, we study the natural action of subgroups $H$ of $\mathrm{Aut}(\mathbf{E})$ on the field $\mathbf{E}$, and analyze how the resulting $H$--orbits encode algebraic information about the extension $\mathbf{E}/\mathbf{E}^H$, where $\mathbf{E}^H$ denotes the fixed field of $H$.

The relevance of focusing on extensions of the form $\mathbf{E}/\mathbf{E}^H$ is explained by a foundational observation (see Proposition~\ref{prop:basic}): a field extension $\mathbf{E}/\mathbf{F}$ is Galois if and only if $\mathbf{F}$ equals $\mathbf{E}^H$ for some subgroup $H$ of $ \mathrm{Aut}(\mathbf{E})$, provided the extension $\mathbf{E}/\mathbf{E}^H$ is algebraic. This motivates the orbit--theoretic study of fixed-field extensions $\mathbf{E}/\mathbf{E}^H$ as a natural framework for understanding Galois theory.

Our first main result establishes a direct criterion for algebraicity over a fixed field: an element $\alpha \in \mathbf{E}$ is algebraic over $\mathbf{E}^H$ if and only if its $H$--orbit is finite. Moreover, in that case, its minimal polynomial coincides with the orbit polynomial, and its degree equals the orbit length (see Theorem~\ref{thm:algebraic-orbit}).

A deeper structural characterization is provided in Theorem~\ref{theo:bounded}, which shows that boundedness of orbit lengths is equivalent to the finiteness of the acting group $H$. In that case, the extension $\mathbf{E}/\mathbf{E}^H$ is primitive --generated by a single element $\theta$ whose orbit has maximal length-- and $H$ coincides with the full Galois group $\mathrm{Gal}(\mathbf{E}/\mathbf{E}^H)$. Moreover, an element $\alpha \in \mathbf{E}$ is a primitive generator of the extension if and only if its orbit has length equal to $[H:1]$.

We then show that $\mathbf{E}/\mathbf{E}^H$ is a Galois extension if and only if all $H$--orbits are finite, and that the extension is finite Galois if and only if the set of orbit lengths is bounded above (see Theorem~\ref{theo:main1}). These results yield a conceptual reformulation of algebraicity, separability, normality, and degree through orbit structure. In particular, Artin’s Lemma, see~\cite[Chapter~VI,~\S1,~Theorem~1.8]{LangAlgebra}, appears as a corollary of this general framework.

In the setting of simple extensions $\mathbf{F}(\alpha)/\mathbf{F}$, we provide an explicit method for computing fixed fields from orbit data. Specifically, when the $H$--orbit of $\alpha$ is finite, say of length $n$, the fixed field $\mathbf{F}(\alpha)^H$ is generated by evaluating the $n$th elementary symmetric polynomials on the orbit $\mathcal{O}_H(\alpha)$ (see Proposition~\ref{prop:fixed-field-symmetric}). This gives an effective algorithmic description of fixed fields, without relying on the full machinery of the Galois correspondence.

The goal of this paper is to develop a unified orbit--theoretic framework in which several classical properties of field extensions are recovered from a single group-theoretic invariant: orbit length. Finiteness of orbits corresponds to algebraicity; boundedness of orbit lengths corresponds to finite degree; and maximal orbit length characterizes primitive elements.

\section{Notation and standing assumptions}

Throughout the paper, $\mathbf{E}$ denotes a field.
For a set $S$, we write $|S|$ for its cardinality.
For a group $G$, the notation $[G:1]$ denotes its order, and for a subgroup
$H\le G$, the index of $H$ in $G$ is written $[G:H]$.
If $\mathbf{F}$ and $\mathbf{E}$ are fields, the notation $\mathbf{F}\le\mathbf{E}$
means that $\mathbf{F}$ is a subfield of $\mathbf{E}$, and the degree of the
extension is denoted by $[\mathbf{E}:\mathbf{F}]$.

We write $\mathrm{Aut}(\mathbf{E})$ for the group of all field automorphisms of
$\mathbf{E}$.
The set of all subfields of $\mathbf{E}$ is denoted by $\mathcal K(\mathbf{E})$,
and the set of all subgroups of\/ $\mathrm{Aut}(\mathbf{E})$ by
$\mathcal H(\mathrm{Aut}(\mathbf{E}))$.

For a field $\mathbf{F}\in\mathcal K(\mathbf{E})$, the \emph{Galois group} of the extension
$\mathbf{E}/\mathbf{F}$ is defined by
\[
\mathrm{Gal}(\mathbf{E}/\mathbf{F})
:=\{\sigma\in\mathrm{Aut}(\mathbf{E})\mid \sigma(a)=a \text{ for all } a\in\mathbf{F}\}.
\]
Conversely, for a group $H\in\mathcal H(\mathrm{Aut}(\mathbf{E}))$, the
corresponding \emph{fixed field} is
\[
\mathbf{E}^H
:=\{\alpha\in\mathbf{E}\mid \sigma(\alpha)=\alpha \text{ for all } \sigma\in H\}.
\]

Both $(\mathcal K(\mathbf{E}),\le)$ and
$(\mathcal H(\mathrm{Aut}(\mathbf{E})),\le)$ are lattices under inclusion.
They are connected by the order--reversing maps
\[
\mathrm{Gal}(\mathbf{E}/\square)\colon
\mathcal K(\mathbf{E})\longrightarrow \mathcal H(\mathrm{Aut}(\mathbf{E})),
\qquad
\mathbf{F}\longmapsto \mathrm{Gal}(\mathbf{E}/\mathbf{F}),
\]
and
\[
\mathbf{E}^{\square}\colon
\mathcal H(\mathrm{Aut}(\mathbf{E}))\longrightarrow \mathcal K(\mathbf{E}),
\qquad
H\longmapsto \mathbf{E}^H.
\]

These maps satisfy the identities
\begin{align}
\mathrm{Gal}(\mathbf{E}/\square)\circ \mathbf{E}^{\square}
\circ \mathrm{Gal}(\mathbf{E}/\square)
&= \mathrm{Gal}(\mathbf{E}/\square), \label{eq:GGG}\\[1mm]
\mathbf{E}^{\square}\circ \mathrm{Gal}(\mathbf{E}/\square)
\circ \mathbf{E}^{\square}
&= \mathbf{E}^{\square}. \label{eq:EEE}
\end{align}

Following \cite[Definition~7.30]{margal} and \cite[Definition~2.15]{Morandi}, a field extension
$\mathbf{E}/\mathbf{F}$ is called \emph{Galois} if it is algebraic and satisfies
\[
\mathbf{E}^{\mathrm{Gal}(\mathbf{E}/\mathbf{F})}=\mathbf{F},
\]
that is,
$\mathbf{E}^{\square}\circ \mathrm{Gal}(\mathbf{E}/\square)(\mathbf{F})=\mathbf{F}$.
\section{Group Actions and Orbit Polynomials}

The orbit--theoretic approach developed in this paper is based on a simple but
fundamental observation: Galois extensions arise as algebraic extensions
determined by fixed fields of automorphism groups.
This principle allows us to shift attention from polynomials and roots to group
actions on fields.

We begin by formulating this observation explicitly.

\begin{proposition}\label{prop:basic}
The Galois extensions $\mathbf{E}/\mathbf{F}$, where
$\mathbf{F}\in\mathcal K(\mathbf{E})$, are exactly the algebraic extensions of
the form $\mathbf{E}/\mathbf{E}^H$, where
$H\in\mathcal H(\mathrm{Aut}(\mathbf{E}))$.
\end{proposition}

\begin{proof}
\textbf{($\Rightarrow$)}
Suppose $\mathbf{E}/\mathbf{F}$ is a Galois extension.
By definition, the extension is algebraic and satisfies
\[
\mathbf{E}^{\square}\circ \mathrm{Gal}(\mathbf{E}/\square)(\mathbf{F})=\mathbf{F}.
\]
Thus $\mathbf{E}/\mathbf{F}$ is of the form $\mathbf{E}/\mathbf{E}^H$ with
$H:=\mathrm{Gal}(\mathbf{E}/\mathbf{F})$.

\medskip
\textbf{($\Leftarrow$)}
Conversely, suppose that the extension $\mathbf{E}/\mathbf{E}^H$ is algebraic for some 
$H\in\mathcal{H}(\mathrm{Aut}(\mathbf{E}))$.
By identity~\eqref{eq:EEE} from Section~2, we have
\[
\mathbf{E}^{\mathrm{Gal}(\mathbf{E}/\mathbf{E}^H)}=\mathbf{E}^H.
\]
Hence $\mathbf{E}/\mathbf{E}^H$ is a Galois extension.
\end{proof}

The above proposition reduces the study of Galois extensions
$\mathbf{E}/\mathbf{F}$ to extensions of the form $\mathbf{E}/\mathbf{E}^H$,
and motivates the following guiding question:
\begin{center}
\emph{Which subgroups $H \in\mathcal{H}(\mathrm{Aut}(\mathbf{E}))$ give rise to algebraic extensions $\mathbf{E}/\mathbf{E}^H$?}
\end{center}
As we will demonstrate, the answer depends entirely on the finiteness
properties of the orbits induced by the action of $H$ on $\mathbf{E}$.

Let $\mathbf{E}$ be a field and let $H \in\mathcal{H}(\mathrm{Aut}(\mathbf{E}))$.
The group $H$ acts naturally on $\mathbf{E}$ by field automorphisms:
\[
H\times \mathbf{E} \longrightarrow \mathbf{E},
\qquad
(\sigma,\alpha)\longmapsto \sigma(\alpha).
\]
This action decomposes $\mathbf{E}$ into disjoint orbits, which encode the symmetry
of elements under the action of $H$.

\begin{definition}
For $\alpha\in\mathbf{E}$, the \emph{$H$-orbit} of $\alpha$ is the set
\[
\mathcal O_H(\alpha):=\{\sigma(\alpha)\mid \sigma\in H\}.
\]
The \emph{orbit length} of $\alpha$ is
\[
\ell(\mathcal O_H(\alpha)):=|\mathcal O_H(\alpha)|.
\]
\end{definition}

Associated with each element $\alpha\in\mathbf{E}$ is its stabilizer subgroup
\[
\mathrm{Stab}_H(\alpha)
:=\{\sigma\in H\mid \sigma(\alpha)=\alpha\}.
\]
By the Orbit--Stabilizer Theorem, see \cite[Theorem~3.19]{rotman1995introduction}, we have
\[
\ell(\mathcal{O}_H(\alpha))=[H:\mathrm{Stab}_H(\alpha)].
\]
In particular, $\alpha\in \mathbf{E}^H$ if and only if
$\ell(\mathcal{O}_H(\alpha))=1$ if and only if the subgroup $H$ equals $\mathrm{Stab}_H(\alpha)$.

We now establish a fundamental link between orbit structure and algebraic
dependence over fixed fields.
The key object is the \emph{orbit polynomial}, whose roots consist precisely
of the elements of the $H$-orbit of a given element.
\begin{definition}
Let $H\in\mathcal{H}(\mathrm{Aut}(\mathbf{E}))$ and let $\alpha\in\mathbf {E}$ be such that
$\ell(\mathcal O_H(\alpha))<\nobreak\infty$.

The \emph{$H$-orbit polynomial} of\/ $\alpha$ is defined by
\[
f_{\alpha,H}(x)
:=\prod_{\beta\in\mathcal O_H(\alpha)}(x-\beta)
\in \mathbf{E}[x].
\]
\end{definition}

By construction, the set of roots of $f_{\alpha,H}(x)$ is exactly the orbit
$\mathcal O_H(\alpha)$.
The following lemma shows that orbit polynomials are naturally defined over the
fixed field.

\begin{lemma}\label{lem:orbitpoly}
If $\alpha\in\mathbf{E}$ has finite $H$-orbit, then the associated orbit polynomial
$f_{\alpha,H}(x)$ lies in the polynomial ring $\mathbf{E}^H[x]$.
\end{lemma}

\begin{proof}
For any $\sigma\in H$, the restriction of $\sigma$ to the finite set
$\mathcal O_H(\alpha)$ permutes its elements.
Hence,
\[
\prod_{\beta\in\mathcal O_H(\alpha)}(x-\beta)
=
\prod_{\beta\in\mathcal O_H(\alpha)}(x-\sigma(\beta)).
\]
It follows that every coefficient of $f_{\alpha,H}(x)$ is fixed by all elements
of $H$, and therefore belongs to the fixed field $\mathbf{E}^H$.
\end{proof}
\begin{remark}[Related Literature]
Orbit--based constructions have long played a role in classical Galois theory,
particularly in the finite case.

For instance, in \cite[Theorem~81]{rotman2001}, Rotman considers a finite Galois extension $\mathbf{E}/\mathbf{F}$ with Galois group $G := \mathrm{Gal}(\mathbf{E}/\mathbf{F})$, and constructs the orbit polynomial $f_{\alpha, G}(x)$ for an element $\alpha \in \mathbf{E}$ that is a root of an irreducible polynomial over $\mathbf{F}$.

A closely related idea also appears in Lang~\cite[Chapter~VI, \S1]{LangAlgebra},  in the proof of Artin’s Lemma (Theorem~1.8), where a finite subgroup $G$ of automorphisms is considered.

The point of view adopted here is different.
Rather than introducing orbit polynomials as a computational tool \emph{after}
algebraicity has been established, or restricting attention to the finite case,
we use orbit structure as a \emph{conceptual criterion} for algebraicity itself.
In this approach, the finiteness of $H$--orbits becomes the primary organizing
principle, providing a direct link between group actions and algebraic
dependence over fixed fields.
\end{remark}
\section{Algebraicity and Finite Orbits}

The main result of this section establishes that an element of 
$\mathbf{E}$ is algebraic over the fixed field $\mathbf{E}^H$
if and only if its $H$-orbit is finite.
Moreover, when this condition holds, the orbit polynomial introduced in the
previous section coincides with the minimal polynomial of $\alpha$
over $\mathbf{E}^H$.
\begin{theorem}\label{thm:algebraic-orbit}
Let $H \in \mathcal{H}(\mathrm{Aut}(\mathbf{E}))$. 

An element $\alpha \in \mathbf{E}$ is algebraic over $\mathbf{E}^H$ if and only if 
\[\ell(\mathcal{O}_H(\alpha)) < \infty.\]

Moreover, in that case, 
\begin{enumerate}
\item[\textnormal{\textbf{(i)}}] The $H$-orbit polynomial $f_{\alpha,H}(x)$ coincides with the minimal polynomial $m_\alpha(x)$ of $\alpha$ over $\mathbf{E}^H$.
\item[\textnormal{\textbf{(ii)}}] $\alpha$ is separable over $\mathbf{E}^H$ and
\item[\textnormal{\textbf{(iii)}}] $[\mathbf{E}^H(\alpha):\mathbf{E}^H]=\ell(\mathcal{O}_H(\alpha))$. 
\end{enumerate}
\end{theorem}

\begin{proof}
($\Rightarrow$) Let $m_\alpha(x)$ be the minimal polynomial of $\alpha$ over $\mathbf{E}^H$.
Since $m_\alpha(x)\in \mathbf{E}^H[x]$, its coefficients are fixed by every
$\sigma\in H$.

For any $\beta\in\mathcal O_H(\alpha)$, there exists $\sigma\in H$ such that
$\beta=\sigma(\alpha)$. Hence
\[
m_\alpha(\beta)
= m_\alpha(\sigma(\alpha))
= \sigma(m_\alpha(\alpha))
= \sigma(0)
= 0.
\]

Thus, every element of $\mathcal{O}_H(\alpha)$ is a root of $m_\alpha(x)$. Since $m_\alpha(x)$ has finitely many roots, it follows that $\ell(\mathcal{O}_H(\alpha)) < \infty$.
\\
($\Leftarrow$) Suppose $\ell(\mathcal{O}_H(\alpha)) < \infty$. Then the $H$-orbit polynomial $f_{\alpha,H}(x)$ is defined and, by Lemma \ref{lem:orbitpoly}, lies in $\mathbf{E}^H[x]$. 

\noindent Moreover, since $f_{\alpha,H}(x)\in \mathbf {E}^H[x]$ and
$f_{\alpha,H}(\alpha)=0$, it follows that $\alpha$ is algebraic over $\mathbf{E}^H$.

\textbf{(i)} Since every element of $\mathcal O_H(\alpha)$ is a root of $m_\alpha(x)$, we have
\[
\deg f_{\alpha,H}(x)\le \deg m_\alpha(x).
\]
On the other hand, the monic polynomial $f_{\alpha,H}(x)\in\mathbf{E}^H[x]$ and vanishes at~$\alpha$.
\\
As $m_\alpha(x)$ is the minimal polynomial of $\alpha$ over $\mathbf{E}^H$, we must have
\[
f_{\alpha,H}(x)=m_\alpha(x).
\]
\indent \textbf{(ii)} Since  $m_\alpha(x)$ equals  $f_{\alpha,H}(x)$ and since all roots of $f_{\alpha,H}(x)$ are distinct, being elements of the orbit
$\mathcal O_H(\alpha)$, it follows that $\alpha$ is separable over $\mathbf{E}^H$.

\textbf{(iii)} 
$
[\mathbf{E}^H(\alpha):\mathbf{E}^H]=\deg m_\alpha(x)=\deg f_{\alpha,H}(x)=\ell(\mathcal{O}_H(\alpha))$.
\end{proof}
Theorem~\ref{thm:algebraic-orbit} provides the precise formulation of the guiding
principle announced in the introduction: algebraicity is equivalent to orbit
finiteness.
\section{Galois Extensions via Orbit Finiteness}

The aim of this section is to show that the fundamental properties of extensions
of the form $\mathbf{E}/\mathbf{E}^H$ are captured by the orbit structure
of the action of $H$ on $\mathbf{E}$. In particular, we will see that the Galois
property corresponds to the finiteness of all $H$-orbits, while the finiteness of
the extension is detected by a uniform bound on orbit lengths.

To formalize this connection, we associate to each 
$H \in \mathcal{H}(\mathrm{Aut}(\mathbf{E}))$
the \emph{set of $H$-orbit lengths}
\[
\mathcal{L}_H := \left\{ \ell(\mathcal{O}_H(\alpha)) \mid \alpha \in \mathbf{E} \right\},
\]
that is, the set of cardinalities of the $H$-orbits of elements
$\alpha \in \mathbf{E}$.

\begin{remark}
Each $H$-orbit has either finite or infinite length. Hence, $\mathcal{L}_H$ is bounded above
if and only if all $H$-orbits are finite. Equivalently, this means that
$\mathcal{L}_H \subseteq \mathbb{N}$ and $\max \mathcal{L}_H$ exists.
\end{remark}
\paragraph{Reduction to the simple case.}
In the proofs that follow, we will occasionally use the Primitive Element
Theorem for finite separable extensions in order to reduce finite extensions
to the simple case.
This use relies only on finiteness and separability and is independent of the
classical Galois correspondence; in particular, it does not presuppose
Artin's Lemma, which will instead be recovered as a consequence of the
orbit--theoretic approach developed here.
\begin{theorem}\label{theo:bounded}
Let $H \in \mathcal{H}(\mathrm{Aut}(\mathbf{E}))$. The following statements are equivalent:
\begin{enumerate}
\item[\textnormal{\textbf{(i)}}] 
The set of $H$-orbit lengths 
$
\mathcal{L}_H
$
is bounded above.
\item[\textnormal{\textbf{(ii)}}] There exists an element $\theta \in \mathbf{E}$, algebraic over $\mathbf{E}^H$, such that
\[
\mathbf{E} = \mathbf{E}^H(\theta).
\]
\item[\textnormal{\textbf{(iii)}}] The group $H$ is finite.
\end{enumerate}

Moreover, in that case:
\begin{enumerate}
\item[\textnormal{\textbf{(iv)}}] $H = \mathrm{Gal}(\mathbf{E}/\mathbf{E}^H)$,
\item[\textnormal{\textbf{(v)}}] $[H : 1] = \ell(\mathcal{O}_H(\theta))$, and
\item[\textnormal{\textbf{(vi)}}] the primitive elements for the extension $\mathbf{E}/\mathbf{E}^H$ are precisely the elements
$\alpha \in \mathbf{E}$ such that
\[
\ell(\mathcal{O}_H(\alpha)) = [H : 1].
\]
\end{enumerate}
\end{theorem}

\begin{proof}
\textbf{(i)} $\Rightarrow$ \textbf{(ii)} Suppose $\mathcal{L}_H$ is bounded above. Then there exists $\theta \in \mathbf{E}$ such that
\[
\ell(\mathcal{O}_H(\theta)) = \max \mathcal{L}_H.
\]
It suffices to show that every element $\alpha\in\mathbf{E}$ lies in
$\mathbf{E}^H(\theta)$.

For any $\alpha \in \mathbf{E}$, consider the tower of fields:
\[
\mathbf{E}^H \leq \mathbf{E}^H(\theta) \leq \mathbf{E}^H(\theta, \alpha),
\]
which implies
\begin{align} 
[\mathbf{E}^H(\theta) : \mathbf{E}^H] \leq [\mathbf{E}^H(\theta, \alpha) : \mathbf{E}^H]. \label{eq:EEEEE56}
\end{align}
By Theorem~\ref{thm:algebraic-orbit}, both $\theta$ and $\alpha$ are algebraic
over $\mathbf{E}^H$ and separable. Hence, we may apply the Primitive Element Theorem, see \cite{brownPrimitive}, to obtain an element $\beta \in \mathbf{E}$ such that
\[
\mathbf{E}^H(\theta, \alpha) = \mathbf{E}^H(\beta).
\]
Thus,
\begin{align}
[\mathbf{E}^H(\theta, \alpha) : \mathbf{E}^H] &= [\mathbf{E}^H(\beta) : \mathbf{E}^H]\nonumber \\
&= \ell(\mathcal{O}_H(\beta)) \leq \ell(\mathcal{O}_H(\theta)) = [\mathbf{E}^H(\theta) : \mathbf{E}^H]. \label{eq:EEEEE57}
\end{align}

Combining (\ref{eq:EEEEE56}) and (\ref{eq:EEEEE57}), we obtain equality throughout, and hence:
\[
\mathbf{E}^H(\theta) = \mathbf{E}^H(\theta, \alpha).
\]
Since this holds for arbitrary $\alpha \in \mathbf{E}$, it follows that $\mathbf{E} = \mathbf{E}^H(\theta)$.
\\
\textbf{(ii)} $\Rightarrow$ \textbf{(iii)}
Since $\theta$ is algebraic over $\mathbf{E}^H$, Theorem~\ref{thm:algebraic-orbit}
yields
\begin{align}
[\mathbf{E}^H(\theta) : \mathbf{E}^H]
=
\ell(\mathcal{O}_H(\theta))
< \infty.
\label{eq:ena}
\end{align}

Consider the $H$-orbit polynomial $f_{\theta,H}(x)$, which exists by
Theorem~\ref{thm:algebraic-orbit}.
This polynomial is separable over $\mathbf{E}^H$, and its splitting field
coincides with $\mathbf{E}^H(\theta)$, since by assumption
$\mathbf{E}=\mathbf{E}^H(\theta)$.
Therefore, by \cite[Theorem~56]{rotman2001}, we have
\begin{align*}
[\mathrm{Gal}(\mathbf{E}^H(\theta)/\mathbf{E}^H):1]
=
[\mathbf{E}^H(\theta):\mathbf{E}^H].
\end{align*}

Since $\mathbf{E}=\mathbf{E}^H(\theta)$, it follows that
\[
\mathrm{Gal}(\mathbf{E}^H(\theta)/\mathbf{E}^H)
=
\mathrm{Gal}(\mathbf{E}/\mathbf{E}^H).
\]
Consequently, $H$ is finite, since it is a subgroup of the finite group
$\mathrm{Gal}(\mathbf{E}/\mathbf{E}^H)$.
\medskip

\textbf{(iii)} $\Rightarrow$ \textbf{(i)}
Since $H$ is finite, every $H$-orbit has cardinality at most $[H:1]$.
Hence the set $\mathcal{L}_H$ is bounded above by $[H:1]$.
\medskip

\textbf{(iv)+(v)}
From the proof of \textbf{(ii)} $\Rightarrow$ \textbf{(iii)} we have
\[
\ell(\mathcal{O}_H(\theta))
=
[\mathrm{Gal}(\mathbf{E}/\mathbf{E}^H):1]
\ge [H:1],
\]
since $H \le \mathrm{Gal}(\mathbf{E}/\mathbf{E}^H)$.
On the other hand, $\ell(\mathcal{O}_H(\theta)) \le [H:1]$.
Therefore,
\[
\ell(\mathcal{O}_H(\theta)) = [H:1],
\text{\ and\ }
H = \mathrm{Gal}(\mathbf{E}/\mathbf{E}^H).
\]
\textbf{(vi)}
An element $\alpha \in \mathbf{E}$ is primitive for the extension
$\mathbf{E}/\mathbf{E}^H$ if and only if $\mathbf{E}^H(\alpha)=\mathbf{E}$, which
is equivalent to
\[
[\mathbf{E}^H(\alpha):\mathbf{E}^H]
=
[\mathbf{E}:\mathbf{E}^H].
\]
By the preceding discussion,
\[
[\mathbf{E}:\mathbf{E}^H]
=
[\mathbf{E}^H(\theta):\mathbf{E}^H]
=
\ell(\mathcal{O}_H(\theta))
=
[H:1].
\]
Finally, by Theorem~\ref{thm:algebraic-orbit},
\[
[\mathbf{E}^H(\alpha):\mathbf{E}^H]
=
\ell(\mathcal{O}_H(\alpha)).
\]
Thus $\alpha$ is primitive if and only if
\[
\ell(\mathcal{O}_H(\alpha)) = [H:1].
\]
\end{proof}
\begin{theorem}\label{theo:main1}
Let $H\in \mathcal{H}(\mathrm{Aut}(\mathbf{E}))$.
\begin{enumerate}
\item[ \textnormal{ \textbf{(i)}}]  
The extension $\mathbf{E} / \mathbf{E}^H$ is Galois if and only if
$
\mathcal{L}_H \subseteq \mathbb{N}
$.
\item[ \textnormal{ \textbf{(ii)}}]
The extension $\mathbf{E} / \mathbf{E}^H$ is finite Galois if and only if
$\mathcal{L}_H$ is bounded above.
\end{enumerate}
\end{theorem}
\begin{proof} 
\textbf{(i)}
By Proposition~\ref{prop:basic}, any algebraic extension of the form
$\mathbf{E}/\mathbf{E}^H$ is Galois.
Therefore, in the present setting, $\mathbf{E}/\mathbf{E}^H$ is Galois
if and only if it is algebraic.
By Theorem~\ref{thm:algebraic-orbit}, $\mathbf{E}/\mathbf{E}^H$ is algebraic
if and only if each orbit $\mathcal{O}_H(\alpha)$ is finite for all
$\alpha \in \mathbf{E}$, that is, if and only if
$\mathcal{L}_H \subseteq \mathbb{N}$.
\\
\textbf{(ii)}
($\Rightarrow$)
Assume that $\mathcal{L}_H$ is bounded above.
Then no $H$-orbit can be infinite, hence $\mathcal{L}_H\subseteq\mathbb{N}$.
By part~\textbf{(i)}, the extension $\mathbf{E}/\mathbf{E}^H$ is Galois.
By Theorem~\ref{theo:bounded}, there exists an element $\theta\in\mathbf{E}$,
algebraic over $\mathbf{E}^H$, such that
\[
\mathbf{E}=\mathbf{E}^H(\theta).
\]
Consequently,
\[
[\mathbf{E}:\mathbf{E}^H]
=
[\mathbf{E}^H(\theta):\mathbf{E}^H]
=
\ell(\mathcal{O}_H(\theta))
<\infty,
\]
and therefore $\mathbf{E}/\mathbf{E}^H$ is a finite Galois extension.
\\
($\Leftarrow$)
Conversely, assume that $\mathbf{E}/\mathbf{E}^H$ is a finite Galois extension.
Then it is finite and separable.
By the Primitive Element Theorem, see \cite{brownPrimitive},
there exists $\theta\in\mathbf{E}$ such that
\[
\mathbf{E}=\mathbf{E}^H(\theta).
\]
Since $\theta$ is algebraic over $\mathbf{E}^H$, Theorem~\ref{thm:algebraic-orbit}
implies that $\ell(\mathcal O_H(\theta))<\infty$.
By Theorem~\ref{theo:bounded}, this is equivalent to the boundedness of
$\mathcal L_H$.
\end{proof}
We may now conclude with the classical result of Artin, as a special case of the above.
\begin{corollary}[Artin's Lemma]\label{cor:Artin}
Let $H \in \mathcal{H}(\mathrm{Aut}(\mathbf{E}))$ be a finite subgroup. Then  $\mathbf{E}/\mathbf{E}^H$ is a finite and Galois extension, with $
\mathrm{Gal}(\mathbf{E}/\mathbf{E}^H) = H$.
\end{corollary}

\begin{proof}
Since $H$ is finite, every $H$-orbit is finite and
$\mathcal L_H$ is bounded above by $[H:1]$.
By Theorem~\ref{theo:main1}, the extension $\mathbf{E}/\mathbf{E}^H$ is Galois.
Moreover, by Theorem~\ref{theo:bounded},
\[
\mathrm{Gal}(\mathbf{E}/\mathbf{E}^H)=H.
\]
\end{proof}
We now illustrate the orbit--theoretic approach to Galois theory through  concrete examples.
\begin{example}[Algebraic Closure of $\mathbb{F}_p$]
Let $p$ be a prime number, and let $\overline{\mathbb{F}}_p$ denote the algebraic
closure of the finite field $\mathbb{F}_p$.

The Frobenius map
\[
\sigma:\overline{\mathbb{F}}_p\longrightarrow\overline{\mathbb{F}}_p,
\qquad
\alpha\longmapsto \alpha^p,
\]
is an element of\/ $\mathrm{Aut}(\overline{\mathbb{F}}_p)$.

Consider the cyclic group $\langle\sigma\rangle$. We claim that the set
\[
\mathcal{L}_{\langle\sigma\rangle}
:=\{\ell(\mathcal{O}_{\langle\sigma\rangle}(\alpha))\mid
\alpha\in\overline{\mathbb{F}}_p\}
\]
is a subset of\/ $\mathbb{N}$ that is unbounded above.

First, observe that the fixed field
$\overline{\mathbb{F}}_p^{\langle\sigma\rangle}$ coincides with $\mathbb{F}_p$.
Indeed, an element $\alpha\in\overline{\mathbb{F}}_p$ is fixed by $\sigma$
if and only if $\alpha^p=\alpha$, which holds precisely for
$\alpha\in\mathbb{F}_p$.

Since $\overline{\mathbb{F}}_p/\mathbb{F}_p$ is an algebraic extension and
$\overline{\mathbb{F}}_p^{\langle\sigma\rangle}=\mathbb{F}_p$, it follows that the orbit length set 
$\mathcal{L}_{\langle\sigma\rangle}$ is a subset of $\mathbb{N}$.
Moreover, $\mathcal{L}_{\langle\sigma\rangle}$ coincides with the set
\[
\{\deg m_\alpha(x)\mid \alpha\in\overline{\mathbb{F}}_p\},
\]
where $m_\alpha(x)$ denotes the minimal polynomial of $\alpha$ over $\mathbb{F}_p$.

Since $\mathbb{F}_p[x]$ contains irreducible polynomials of arbitrarily large
degree, the set $\mathcal{L}_{\langle\sigma\rangle}$ is unbounded above.

Finally, by Theorem~\ref{theo:bounded}, the unboundedness of
$\mathcal{L}_{\langle\sigma\rangle}$ implies that the group
$\langle\sigma\rangle$ is infinite.
\end{example}
\begin{example}[Primitive Elements]
Let $\{p_i \mid 1 \le i \le n\}$ be a set of $n$ distinct prime numbers, and define
\[
\sqrt{P}:=\{\sqrt{p_i}\mid 1\le i\le n\},
\qquad
\sqrt{\widehat{P}}_i:=\sqrt{P}\setminus\{\sqrt{p_i}\}.
\]

Consider the field extension
\[
\mathbb{Q}(\sqrt{P})
:=\mathbb{Q}(\sqrt{p_1},\sqrt{p_2},\dots,\sqrt{p_n}).
\]
We claim the following:

\textit{
The element
\[
\theta:=\sum_{i=1}^n \sqrt{p_i}\in \mathbb{Q}(\sqrt{P})
\]
is a primitive element for the extension $\mathbb{Q}(\sqrt{P})/\mathbb{Q}$ and  $
\mathrm{Gal}(\mathbb{Q}(\sqrt{P})/\mathbb{Q})$ is isomorphic to the group $(\mathbb{Z}/2\mathbb{Z})^n$.}

According to Roth~\cite[Corollaries~1 and~2]{Roth1971}, for each $i, 1\leq i\leq n,$ the element
$\sqrt{p_i}$ does not lie in $\mathbb{Q}(\sqrt{\widehat{P}}_i)$, and the degree of
the extension $\mathbb{Q}(\sqrt{P})/\mathbb{Q}$ is $2^n$.

Since
\[
\mathbb{Q}(\sqrt{P})
=\mathbb{Q}(\sqrt{\widehat{P}}_i)(\sqrt{p_i}),
\]
and the minimal polynomial of $\sqrt{p_i}$ over
$\mathbb{Q}(\sqrt{\widehat{P}}_i)$ is $x^2-p_i$, there exist $n$ distinct field
automorphisms
\[
\sigma_i\in\mathrm{Aut}(\mathbb{Q}(\sqrt{P})),\qquad 1\le i\le n,
\]
satisfying
\[
\sigma_i(\sqrt{p_i})=-\sqrt{p_i},
\qquad
\sigma_i(q)=q \quad \text{for all } q\in\mathbb{Q}(\sqrt{\widehat{P}}_i).
\]

Let $H$ be the group generated by the automorphisms $\sigma_1,\dots,\sigma_n$.
Since each $\sigma_i$ has order~$2$ and
$\sigma_i\circ\sigma_j=\sigma_j\circ\sigma_i$ for all $i,j$,
the group $H$ is abelian and isomorphic to $(\mathbb{Z}/2\mathbb{Z})^n$.
In particular, $
[H:1]=2^n
$.

By Corollary~\ref{cor:Artin}, the extension
$\mathbb{Q}(\sqrt{P})/\mathbb{Q}(\sqrt{P})^H$ is Galois, with
\[
\mathrm{Gal}(\mathbb{Q}(\sqrt{P})/\mathbb{Q}(\sqrt{P})^H)=H
\cong (\mathbb{Z}/2\mathbb{Z})^n
\]
and
\[
[\mathbb{Q}(\sqrt{P}):\mathbb{Q}(\sqrt{P})^H]=[H:1]=2^n.
\]

The $H$-orbit of
\[
\theta:=\sum_{i=1}^n\sqrt{p_i}
\]
is
\[
\mathcal{O}_H(\theta)
=\{\pm\sqrt{p_1}\pm\sqrt{p_2}\pm\dots\pm\sqrt{p_n}\},
\]
so $\ell(\mathcal{O}_H(\theta))=[H:1]$.
\\
By Theorem~\ref{theo:bounded}, the element $\theta$ is primitive for the extension
\[\mathbb{Q}(\sqrt{P})/\mathbb{Q}(\sqrt{P})^H.\]

Finally, since
\[
\mathbb{Q}\le\mathbb{Q}(\sqrt{P})^H\le\mathbb{Q}(\sqrt{P})
\ \text{and}\ 
[\mathbb{Q}(\sqrt{P}):\mathbb{Q}]=2^n
=[\mathbb{Q}(\sqrt{P}):\mathbb{Q}(\sqrt{P})^H],
\]
we conclude that $\mathbb{Q}(\sqrt{P})^H=\mathbb{Q}$. Therefore,  $\mathbb{Q}(\sqrt{P})/\mathbb{Q}$ is a Galois extension with 
\[
\mathrm{Gal}(\mathbb{Q}(\sqrt{P})/\mathbb{Q})\cong (\mathbb{Z}/2\mathbb{Z})^n
\]
and primitive element $\theta=\sum_{i=1}^n\sqrt{p_i}$.
\end{example}
\section{An Application: Fixed Fields via Symmetric Polynomials}
We conclude by showing that the orbit data developed above allow fixed fields to
be computed explicitly.
In the case of a simple extension, this computation reduces to elementary
symmetric polynomials on a single orbit, see~\cite[Proposition~8.24]{margal}.

\begin{proposition}\label{prop:fixed-field-symmetric}
Let\/ $\mathbf{F}(\alpha)/\mathbf{F}$ be a simple field extension, and let
$H$ be a subgroup of the Galois group $\mathrm{Gal}(\mathbf{F}(\alpha)/\mathbf{F})$.

Suppose that the $H$-orbit of $\alpha$ is finite, say
\[
\mathcal O_H(\alpha)=\{\alpha_1,\alpha_2,\dots,\alpha_n\}.
\]
Then the fixed field $\mathbf{F}(\alpha)^H$ coincides with
\[
\mathbf{F}\bigl(
\varepsilon_1(\alpha_1,\dots,\alpha_n),
\dots,
\varepsilon_n(\alpha_1,\dots,\alpha_n)
\bigr),
\]
where $\varepsilon_i, 1\leq i\leq n,$ denotes the $i$th elementary symmetric polynomial.
\end{proposition}

\begin{proof}
Let 
\[
f_{\alpha,H}(x) = \prod_{\beta \in \mathcal{O}_H(\alpha)} (x - \beta)= x^n + \kappa_{n-1}x^{n-1} + \cdots + \kappa_0
\]
 be the $H$-orbit polynomial of $\alpha$, where 
\[
\mathcal{O}_H(\alpha) = \{ \alpha_1, \dots, \alpha_n \}.
\] 
By Theorem~\ref{thm:algebraic-orbit}, the polynomial $f_{\alpha,H}(x)$ coincides
with the minimal polynomial of $\alpha$ over $\mathbf{F}(\alpha)^H$. Therefore,
\begin{align} 
[\mathbf{F}(\alpha)^H(\alpha) : \mathbf{F}(\alpha)^H] = n.\label{eq:EEE3}
\end{align}
By Vi\`{e}te's formulas, the coefficient $\kappa_{n-i}$  of $f_{\alpha,H}(x)$, for $1 \leq i \leq n$, is given by
 $(-1)^i \varepsilon_i(\alpha_1, \dots, \alpha_n)$, where $\varepsilon_i$ is the $i$th elementary symmetric polynomial. Hence, $\varepsilon_i(\alpha_1, \dots, \alpha_n)\in \mathbf{F}(\alpha)^H$, and therefore we may form 
the tower of fields:
\begin{align} 
\mathbf{F}\big( \varepsilon_1(\alpha_1, \dots, \alpha_n), \dots, \varepsilon_n(\alpha_1, \dots, \alpha_n) \big)
 \leq\mathbf{F}(\alpha)^H  \leq\ \mathbf{F}(\alpha)^H(\alpha) = \mathbf{F}(\alpha).\label{eq:EEEtower}
\end{align}
Since $f_{\alpha,H}(x)$ lies in the polynomial ring 
\[
\mathbf{F}\big( \varepsilon_1(\alpha_1, \dots, \alpha_n), \dots, \varepsilon_n(\alpha_1, \dots, \alpha_n) \big)[x]
\]
and vanishes at $\alpha$, it follows that
\begin{align} 
[\mathbf{F}(\alpha) : \mathbf{F}\big( \varepsilon_1(\alpha_1, \dots, \alpha_n), \dots, \varepsilon_n(\alpha_1, \dots, \alpha_n) \big)] \leq n.\label{eq:EEE4}
\end{align}
By the tower of fields (\ref{eq:EEEtower}) and relations (\ref{eq:EEE3}) and (\ref{eq:EEE4}) we obtain:
\[
\begin{aligned}
[\mathbf{F}(\alpha)^H(\alpha) : \mathbf{F}(\alpha)^H]
&= n \\
&= [\mathbf{F}(\alpha)^H(\alpha) :
\mathbf{F}\big( \varepsilon_1(\alpha_1, \dots, \alpha_n), \dots,
\varepsilon_n(\alpha_1, \dots, \alpha_n) \big)].
\end{aligned}
\]

Hence, the intermediate field $\mathbf{F}(\alpha)^H$ must coincide with 
\[
\mathbf{F}\big( \varepsilon_1(\alpha_1, \dots, \alpha_n), \dots, \varepsilon_n(\alpha_1, \dots, \alpha_n) \big),
\]
as desired.
\end{proof}

The proposition above shows that the fixed field
$\mathbf{F}(\alpha)^H$ admits an explicit description in terms of symmetric
expressions evaluated on the $H$-orbit of the single element $\alpha$.

\subsection*{A classical example}

We illustrate Proposition~\ref{prop:fixed-field-symmetric} with a classical
example originating in Artin, see \cite[Section~II.G]{ArtinGalois} and \cite[Example~8.28]{margal}. Our presentation emphasizes the
orbit--theoretic viewpoint and the role of orbit polynomials in describing
the fixed field.

Let us consider the simple extension $\mathbf{F}(t)/\mathbf{F}$, where
$\mathbf{F}(t)$ denotes the field of rational functions over $\mathbf{F}$.

The automorphisms
\begin{align*}
\sigma_0(t)&=t, & \sigma_1(t)&=1-t, & \sigma_2(t)&=\frac1t,\\
\sigma_3(t)&=1-\frac1t, 
& \sigma_4(t)&=\frac1{1-t}, 
& \sigma_5(t)&=\frac{t}{t-1}.
\end{align*}
are elements of the Galois group $\mathrm{Gal}(\mathbf{F}(t)/\mathbf{F})$.

The set
\[
H=\{\sigma_0,\sigma_1,\sigma_2,\sigma_3,\sigma_4,\sigma_5\}
\]
is a subgroup of 
$\mathrm{Gal}(\mathbf{F}(t)/\mathbf{F})$ with generators $\sigma_1$ and $\sigma_2$.

The $H$-orbit of the primitive element $t$ is
\[
\mathcal O_H(t)
=\left\{
t,\ 1-t,\ \frac1t,\ 1-\frac1t,\ \frac1{1-t},\ \frac{t}{t-1}
\right\}.
\]
Let 
\[
\alpha_i
:=\varepsilon_i\!\left(
t,\ 1-t,\ \frac1t,\ 1-\frac1t,\ \frac1{1-t},\ \frac{t}{t-1}
\right), \quad  1\leq i\leq 6,
\]
denote the evaluation of  the elementary symmetric polynomial $\varepsilon_i(x_1, x_2, \dots, x_6)$ on the elements of $\mathcal O_H(t)$.
  
A direct computation shows that
\[
\alpha_1=\alpha_5=3,\quad
\alpha_6=1, \quad \alpha_3=-\frac{2-6 t+5 t^2+5 t^4-6 t^5+2 t^6}{(-1+t)^2 t^2}
\]
and
\[
\alpha_2=\alpha_4
=-\frac{1-3t+5t^3-3t^5+t^6}{(t-1)^2t^2}.
\]
Moreover,
\[
\alpha_3=
\begin{cases}
2\alpha_2-5, & \text{if }\mathrm{char}(\mathbf{F})\neq 2,\\[2mm]
1, & \text{if }\mathrm{char}(\mathbf{F})=2.
\end{cases}
\]
Thus,
\[
\mathbf{F}(t)^H=\mathbf{F}(\alpha_2),
\]
recovering Artin's classical computation using only orbit data.

As shown in \cite[Section~II.G]{ArtinGalois}, the fixed field $\mathbf{F}(t)^H$ is equal to 
\[
\mathbf{F}\left(\frac{(t^2 - t + 1)^3}{(t-1)^2 t^2}\right).
\]
Note that
\[
\alpha_2=-\frac{1 - 3t + 5t^3 - 3t^5 + t^6}{(t-1)^2 t^2}
= 6-\frac{(t^2 - t + 1)^3}{(t-1)^2 t^2} ,
\]
so $\alpha_2$ differs from Artin's generator by a constant.
\begin{remark}
Proposition~\ref{prop:fixed-field-symmetric} provides an explicit and effective
method for determining fixed subfields of finite Galois extensions using only
orbit data.

Indeed, let $\mathbf{E}/\mathbf{F}$ be a finite Galois extension with Galois group
$G := \mathrm{Gal}(\mathbf{E}/\mathbf{F})$.
Then $\mathbf{F} = \mathbf{E}^G$, and by the Primitive Element Theorem there exists
an element $\theta \in \mathbf{E}$ such that
\[
\mathbf{E} = \mathbf{E}^G(\theta).
\]
Given a subgroup $H \le G$, the fixed field $\mathbf{E}^H$ can be determined by
computing the finite $H$-orbit of $\theta$ and evaluating the elementary
symmetric polynomials on this orbit.
The resulting expressions generate $\mathbf{E}^H$.

In this way, fixed fields are recovered directly from orbit data, without
invoking the full Galois correspondence.
\end{remark}
\section{Conclusion}

The orbit structure of subgroup actions on fields provides a unifying invariant for analyzing fundamental properties of field extensions. In this paper, we have shown that the finiteness of $H$--orbits corresponds precisely to algebraicity over the fixed field $\mathbf{E}^H$, and that orbit length coincides with the degree of the minimal polynomial. This correspondence allows classical notions such as separability and primitivity to be expressed in purely group--theoretic terms.

At the level of extensions, we proved that the extension $\mathbf{E}/\mathbf{E}^H$ is Galois if and only if all $H$--orbits are finite, and that it is finite Galois precisely when the set of orbit lengths is bounded above. In this setting, Artin’s Lemma appears as a special case of the orbit–finiteness principle, and primitive elements are characterized by maximal orbit length.

Finally, we provided a constructive application: when $\mathbf{E} = \mathbf{F}(\alpha)$ and the $H$--orbit of $\alpha$ is finite, the fixed field $\mathbf{E}^H$ is explicitly generated by evaluating elementary symmetric polynomials on the orbit. This yields an effective algorithm for computing fixed fields directly from group action data, independent of the full Galois correspondence.

Together, these results recast key aspects of Galois theory through the lens of group actions and orbit finiteness, offering both conceptual clarity and computational access to classical invariants.

  

\end{document}